\documentclass[12pt]{amsart}

\usepackage{amscd}
\usepackage{amsfonts}
\usepackage{amssymb}
\usepackage{amsmath}
\usepackage{color}

\numberwithin{equation}{section}

\newtheorem{pro}{Proposition}[section]
\newtheorem{lemma}[pro]{Lemma}
\newtheorem{theorem}[pro]{Theorem}
\newtheorem{definition}[pro]{Definition}
\newtheorem{remark}[pro]{Remark}

\newtheorem{corollary}[pro]{Corollary}
\newtheorem{problem}[pro]{Problem}
\newtheorem{conjecture}[pro]{Conjecture}

\newtheorem*{proA*}{Proposition A}
\newtheorem*{thmA*}{Theorem A}
\newtheorem*{thmB*}{Theorem B}
\newtheorem*{coroB*}{Corollary B}
\newtheorem*{coroC*}{Corollary C}
\newtheorem*{thmC*}{Theorem C}
\newtheorem*{coroD*}{Corollary D}
\newtheorem*{thmD*}{Theorem D}

\newtheorem*{fact1*}{Fact 1}

\newcommand{\al}{\alpha}

\newcommand{\ga}{\gamma}

\newcommand{\lam}{\lambda}

\newcommand{\ra}{{\rightarrow}}
\newcommand{\lra}{{\longrightarrow}}

\newcommand{\ie}{{i.e$.$\,}}
\newcommand{\cf}{{cf$.$\,}}

\newcommand{\CC}{{\mathbb C}}

\newcommand{\HH}{{\mathbb H}}

\newcommand{\NN}{{\mathbb N}}
\newcommand{\PP}{{\mathbb P}}
\newcommand{\RR}{{\mathbb R}}

\newcommand{\ZZ}{{\mathbb Z}}

\newcommand{\rA}{\mathop{\rm A}\nolimits}
\newcommand{\rE}{\mathop{\rm E}\nolimits}
\newcommand{\rO}{\mathop{\rm O}\nolimits}

\newcommand{\rSim}{\mathop{\rm Sim}\nolimits}

\newcommand{\rPO}{\mathop{\rm PO}\nolimits}

\newcommand{\rGL}{\mathop{\rm GL}\nolimits}

\newcommand{\PO}{\mathop{\rm PO}\nolimits}

\newcommand{\Sim}{\mathop{\rm Sim}\nolimits}
\newcommand{\Aut}{\mathop{\rm Aut}\nolimits}

\newcommand{\dev}{\mathop{\rm dev}\nolimits}

\newcommand{\cR}{{\mathcal{R}}}
\newcommand{\cN}{{\mathcal{N}}}

\newcommand{\cI}{{\mathcal{I}}}
\newcommand{\cS}{{\mathcal{S}}}

\newcommand{\fL}{{\mathsf{L}}}

\newcommand{\fR}{{\mathsf{R}}}
\newcommand{\fS}{{\mathsf{S}}}

\begin{document}
\baselineskip 13pt 
\thispagestyle{empty}

\title[\ ]
{Lorentzian  similarity manifold}

\author[\ ]{Yoshinobu KAMISHIMA}
\address{Department of Mathematics, Tokyo Metropolitan
University,  Minami-Ohsawa 1-1, Hachioji, Tokyo 192-0397, Japan.}
\email{kami@tmu.ac.jp}

\thanks{This research was initiated
during the stay at ESI in July 2011. He gratefully acknowledges the
support of the University of Vienna via the ESI}
\date{\today}


\begin{abstract} If an $m+2$-manifold $M$ is locally modeled on
$\RR^{m+2}$ with coordinate changes lying in the subgroup
$G=\RR^{m+2}\rtimes ({\rO}(m+1,1)\times \RR^+)$ of the affine group
${\rA}(m+2)$, then $M$ is said to be a \emph{Lorentzian similarity
manifold}. A Lorentzian similarity manifold is also a conformally
flat Lorentzian manifold because $G$ is isomorphic to the stabilizer
of the Lorentz group ${\rPO}(m+2,2)$ which is the full Lorentzian
group of the Lorentz model $S^{2n+1,1}$. It contains a class of
Lorentzian flat space forms. We shall discuss the properties of
compact Lorentzian similarity manifolds using developing maps and
holonomy representations.
\end{abstract}

\maketitle

\setcounter{section}{0}
\section{Introduction}
Let ${\rA}(m+2)=\RR^{m+2}\rtimes {\rGL}(m+2,\RR)$ be the affine
group of the $m+2$-dimensional euclidean space $\RR^{m+2}$. An
$m+2$-manifold $M$ is an \emph{affinely flat} manifold if $M$ is
 locally modeled on
$\RR^{m+2}$ with coordinate changes lying in ${\rA}(m+2)$.
 When $\RR^{m+2}$
is endowed with a Lorentz inner product, we obtain \emph{Lorentz
similarity geometry}
$${\rSim_L}(\RR^{m+2})=\RR^{m+2}\rtimes ({\rO}(m+1,1)\times
\RR^+)$$ as a subgeometry of ${\rA}(m+2)$. If an affinely flat
manifold $M$ is locally modeled on ${\rSim_L}(\RR^{m+2})$, then $M$
is said to be a \emph{Lorentzian similarity manifold}. Lorentzian
similarity geometry contains
 \emph{Lorentzian flat geometry} $({\rE}(m+1,1),\RR^{m+2})$ where
$\displaystyle {\rm E}(m+1,1)=\RR^{m+2}\rtimes {\rm O}(m+1,1)$.

\begin{thmA*}\label{lsim1}
If $M$ is a compact complete Lorentzian similarity manifold, then
$M$ is a Lorentzian flat space form. Furthermore,
$M$ is diffeomorphic to an infrasolvmanifold.
\end{thmA*}

Theorem A is proved as follows (\cf Section \ref{A}); The fundamental group $\pi_1(M)$
of a compact complete Lorentzian similarity manifold $M$ is shown to be virtually solvable.
Then we prove that $\pi_1(M)$ admits a nontrivial translation subgroup. Using
these results, $M$ will be a compact complete Lorentzian flat manifold.
In particular, the Auslander-Milnor conjecture is true for
compact complete Lorentzian similarity manifolds (\cf \cite{mi}).\\
Let $({\rm PO}(m+2,2), S^{m+1,1})$ be \emph{conformally flat
Lorentzian geometry}. If a point ${\hat\infty}\in S^{m+1,1}$ is
defined as the projectivization of a null vector in $\RR^{m+4}$, the
stabilizer ${\rPO}(m+2,2)_{\hat\infty}$ is isomorphic to
${\rSim_L}(\RR^{m+2})$ for which there is a suitable conformal
Lorentzian embedding of $\RR^{m+2}$ into $S^{m+1,1}-\{\hat\infty\}$
which is equivariant with respect to
${\rSim_L}(\RR^{m+2})={\rPO}(m+2,2)_{\hat\infty}$ (\cf \cite{kam1}).
In contrast to \emph{conformally flat Riemannian geometry},
$\RR^{m+2}$ is properly contained in the complement
$S^{m+1,1}-\{\hat\infty\}$ (\cf \cite{bcdgm}). A Lorentzian
similarity geometry $({\rSim_L}(\RR^{m+2}),\RR^{m+2})$ is a sort of
subgeometry of conformally flat Lorentzian geometry $({\rm
PO}(m+2,2), S^{m+1,1})$.

In general, the structure group of a conformally flat Lorentzian
manifold belongs to ${\rO}(m+1,1)\times\RR^+$. Let
${\rSim}^*(\RR^{m})=\RR^{m}\rtimes ({\rO}(m)\times \RR^*)$ be the
similarity subgroup of ${\rO}(m+1,1)$.\\
 Take a subgroup
${\rSim}^*(\RR^{m})\times \RR^+$ in ${\rO}(m+1,1)\times\RR^+$. We
call $M$ a \emph{conformally flat Lorentzian parabolic} manifold if
the structure group is conjugate to a subgroup of
${\rSim}^*(\RR^{m})\times \RR^+$. (See Definition \ref{FL}.) In
Section \ref{para} we prove (\cf Theorem \ref{Lsimilaritymanifold1})
\begin{thmB*}\label{Lsimilaritymanifold}
Let $M$ be an $m+2$-dimensional compact conformally flat
Lorentzian manifold whose holonomy group is virtually solvable
in ${\rSim_L}(\RR^{m+2})$.
Then $M$ is either a conformally flat Lorentzian
parabolic manifold or finitely covered by the Lorentz model
 $S^1\times S^{m+1}$, a Hopf manifold $S^{m+1}\times S^1$,
or a torus $T^{m+2}$.
\end{thmB*}

For $m=2n$, there is the natural embedding $\displaystyle {\rm
U}(n+1,1)\ra {\rm O}(2n+2,2)$ so that $({\rm U}(n+1,1),S^1\times
S^{2n+1})$ is a subgeometry of $({\rm O}(2n+2,2), S^1\times
S^{2n+1})$. Here $S^1\times S^{2n+1}$ is a two-fold covering of
$S^{2n+1,1}$. A $2n+2$-dimensional manifold $M$ is said to be a
\emph{conformally flat Fefferman-Lorentz parabolic} manifold if $M$
is uniformized with respect to $({\rm U}(n+1,1),S^1\times
S^{2n+1})$. (Compare \cite{kam6}.) We study which compact
conformally flat Fefferman-Lorentz parabolic manifolds are the
quotients of domains of $S^{m+1,1}-\{\hat\infty\}$ by properly
discontinuous subgroups of ${\rPO}(m+2,2)_{\hat\infty}$ in Section
\ref{holonomydiscrete}. See \cite{kam4} for a related work.

\begin{thmC*}\label{discrete*}
Let $M$ be a $2n+2$ -dimensional compact conformally flat
Fefferman-Lorentz parabolic
manifold and
$$(\rho,\dev):(\pi_1(M),\tilde M)\ra
({\rm U}(n+1,1)^{\sim},\RR\times S^{2n+1})$$ the developing pair.
Suppose that the holonomy group $\Gamma$ is discrete in ${\rm U}(n+1,1)^{\sim}$.
If the developing map ${\dev}:\tilde M\ra\tilde S^{2n+1,1}$
misses a closed subset which is invariant under $\RR$ and $\Gamma$,
then ${\dev}$ is a covering map onto its image.
\end{thmC*}

For noncompact complete Lorentzian case, \ie, properly discontinuous
actions of free groups on complete simplly connected Lorentzian flat
manifolds, see \cite{cdg}, \cite{glm},
 \cite{bcdgm} for details.

\section{Lorentzian similarity manifold}\label{A}

Consider the following exact sequence:
\begin{equation}\label{holexact}
1\ra \RR^{m+2}\rtimes \RR^+\ra {\rSim_L}(\RR^{m+2})\stackrel{P}\lra
{\rO}(m+1,1)\ra 1.
\end{equation}

\begin{lemma}\label{lsim11}
Let $M=\RR^{m+2}/\Gamma$
be a compact complete Lorentzian similarity manifold where
$\Gamma\leq {\rSim_L}(\RR^{m+2})$. Suppose that
$P(\Gamma)$ is discrete in ${\rO}(m+1,1)$.
If $\Delta=(\RR^{m+2}\rtimes \RR^+)\cap \Gamma$, then $\Delta\leq \RR^{m+2}$
which is nontrivial.
\end{lemma}

\begin{proof}
Since $P(\Gamma)$ is discrete, it acts properly discontinuously
on the $m+1$-dimensional hyperbolic space $\HH^{m+1}_\RR=
{\rO}(m+1)\times{\rO}(1)\backslash{\rO}(m+1,1)$.
The (virtually) cohomological dimension ${\rm vcd}$ of
$P(\Gamma)$ satisfies ${\rm vcd}(P(\Gamma))\leq m+1$.
On the other hand, the cohomological dimension ${\rm cd}(\Gamma)=m+2$,
the intersection $\Delta$ of \eqref{holexact} is nontrivial.
Let $$
1\ra \RR^{m+2}\ra \RR^{m+2}\rtimes \RR^+\stackrel{p}\lra \RR^+\ra 1$$
be the exact sequence. If $p(\Delta)$ is nontrivial, then we may assume
that there exists an element $\ga=(a,\lambda)\in \Delta$ such that
$p(\ga)=\lam<1$.
A calculation shows $\displaystyle \ga^n=(\frac {1-\lam^n}{1-\lam}a,\lam^n)$
$(\forall\, n\in \ZZ)$. The sequence of the orbits $\{\ga^n\cdot 0;\, n\in\ZZ\}$
at the origin $0\in \RR^{m+2}$ converges when $n\ra \infty$,
\[
\ga^n\cdot 0=\frac {1-\lam^n}{1-\lam}a+\lam^n\cdot 0=\frac {1-\lam^n}{1-\lam}a\lra
\frac {1}{1-\lam}a. \]
As $\Delta$ acts properly discontinuously on $\RR^{m+2}$,
$\{\ga^n;\, n=1,2,\dots\}$ is a finite set. Since $\Delta$ is torsionfree, $\ga=1$ which is
a contradiction. So $p(\Gamma)$ must be trivial.

\end{proof}

\begin{pro}\label{lsim12}
Let $M=\RR^{m+2}/\Gamma$
be a compact complete Lorentzian similarity manifold.
Then $\Gamma$ is virtually solvable in ${\rSim_L}(\RR^{m+2})$.
\end{pro}

\begin{proof}
\noindent {\bf (1)}\, When $P(\Gamma)$ is discrete, we obtain the following exact sequences
from \eqref{holexact}.

{\scriptsize
\begin{equation}\label{discrecase}
\begin{CD}
1@>>> \RR^{m+2}@>>> {\rSim_L}(\RR^{m+2})@>L>>{\rO}(m+1,1)\times\RR^+@>>> 1\\
@. @AAA  @AAA @AAA @. \\
1@>>> \Delta@>>> \Gamma@>L>>L(\Gamma)@>>> 1\\
\end{CD}
\end{equation}
}
If $\Delta\cong \ZZ^k$, then the span $\RR^k$ of $\Delta$ in $\RR^{m+2}$ is normalized
by $\Gamma$. Let $\langle\,,\rangle$ be the Lorentz inner product on $\RR^{m+2}$.
The rest of the argument is similar to that of \cite{gk}.
In fact, $L(\Gamma)$ of \eqref{discrecase}
induces a properly discontinuous affine
action $\rho$ on $\RR^{m+2-k}$ with finite kernel
${\rm Ker}\,\rho$:
$$\rho: L(\Gamma)\ra {\rm Aff}(\RR^{m+2-k}).$$
(Compare Lemma \ref{properaff}.) If necessary, we can find a
torsionfree normal subgroup of finite index in $\rho(L(\Gamma))$ by
Selberg's lemma. Passing to a finite index subgroup if necessary,
the quotient $\RR^{m+2-k}/\rho(L(\Gamma))$ is a
compact complete affinely flat manifold.\\

Suppose that $\langle\,,\rangle\,|_{\RR^k}$ is nondegenerate.
According to whether $\langle\,,\rangle\,|_{\RR^k}$ is positive
definite or
 indefinite, $\RR^{m+2-k}/\rho(L(\Gamma))$ is a compact  complete
Lorentzian similarity manifold or Riemannian similarity manifold respectively.

If $\RR^{m+2-k}/\rho((L(\Gamma))$ is a Lorentzian similarity
manifold, by induction hypothesis, $L(\Gamma)$ is virtually
solvable. When $\RR^{m+2-k}/\rho((L(\Gamma))$ is a Riemannian
similarity manifold, \ie $\rho((L(\Gamma))\leq {\rm
Sim}(\RR^{m+2-k})$ which is an amenable Lie group, a discrete
subgroup $\rho((L(\Gamma))$ is virtually solvable by Tits' theorem.
(Compare \cite{mi}. Furthermore, $\RR^{m+2-k}/\rho((L(\Gamma))$ is a
Riemannian flat manifold by Fried's theorem \cite{df}.) In each
case, $\Gamma$ is virtually solvable. \\

If $\langle\,,\rangle\,|_{\RR^k}$ is degenerate, then $\RR^k=\fR$
consisting of a lightlike vector as a basis. The holonomy group
$L(\Gamma)$ leaves invariant $\fR$. The subgroup of ${\rm
O}(m+1,1)\times \RR^+$ preserving $\fR$ is isomorphic to ${\rm
Sim}^*(\RR^{m})\times \RR^+= (\RR^m\rtimes({\rO}(m)\times
\RR^*))\times \RR^+$ which is an amenable Lie group. As
$L(\Gamma)\leq{\rm Sim}^*(\RR^{m})\times \RR^+$,
$L(\Gamma)$ is virtually solvable so is $\Gamma$.\\

\noindent {\bf (2)}\, When $P(\Gamma)$ is indiscrete, it follows
from \cite[Theorem 8.24]{ra}
that the identity component of the closure $\overline{P(\Gamma)}^0$ is
solvable in ${\rm O}(m+1,1)$. It belongs to the maximal amenable subgroup
up to conjugate:
\[
\overline{P(\Gamma)}^0\leq \RR^m\rtimes({\rO}(m)\times \RR^*).\] It
is easy to check that the normalizer of $\overline{P(\Gamma)}^0$ is
still contained in $\RR^m\rtimes({\rO}(m)\times \RR^*)$ because the
normalizer leaves invariant at most two points $\{0,\infty\}$  on
the boundary $S^{m}=\partial \HH^{m+1}_\RR$ for which ${\rm
O}(m+1,1)_\infty= \RR^m\rtimes({\rO}(m)\times \RR^*)$. Hence
$P(\Gamma)\leq \RR^m\rtimes({\rO}(m)\times \RR^*)$. There is an
exact sequence induced from \eqref{holexact}:
\begin{equation*}
1\ra \RR^{m+2}\rtimes \RR^+\ra P^{-1}(\RR^m\rtimes({\rO}(m)\times \RR^*))
\stackrel{P}\lra
\RR^m\rtimes({\rO}(m)\times \RR^*)\ra 1
\end{equation*}in which $P^{-1}(\RR^m\rtimes({\rO}(m)\times \RR^*))$ is
an amenable Lie subgroup.
Hence, $\Gamma$ is virtually solvable.

\end{proof}

\begin{pro}\label{lsim13}
Let $M$ be a compact complete Lorentzian similarity manifold $\RR^{m+2}/\Gamma$.
Then $M$ is diffeomorphic to an infrasolvmanifold $U/\Gamma$.
\end{pro}
\begin{proof}
As $\Gamma\leq \RR^{m+2}\rtimes({\rO}(m+1,1)\times \RR^+)$ is a
virtually solvable group, take the real algebraic hull
$A(\Gamma)=U\cdot T$ where $U$ is a unipotent radical and $T$ is a
reductive $d$-subgroup such that $T/T^0$ is finite. Then each
element $r=u\cdot t\in U\cdot T$ acts on $U$ by $\ga x=utxt^{-1}$
$(x\in U)$. It follows from the result of \cite{ol} that $\Gamma$
acts properly discontinuously on $U$ such that $U/\Gamma$ is
compact. Furthermore $U/\Gamma$ is diffeomorphic to an
infrasolvmanifold
by \cite[Theorem 1.2]{ol}.\\

Since $U/\Gamma$ is compact, we choose a compact subset $D\subset U$
such that $U=\Gamma\cdot D$. As $\Gamma$ acts properly
discontinuously on $\RR^{m+2}$ and $U\cdot T\leq
\RR^{m+2}\rtimes({\rO}(m+1,1)\times \RR^+)$, it is easily checked
that $U$ acts properly on $\RR^{m+2}$. Since $T$ is reductive, we
may assume that $T\cdot 0=0\in \RR^{m+2}$. Define a map;
\[\rho:U\ra \RR^{m+2},\ \ \rho(x)=x\cdot 0.\]
Noting that $U$ acts freely on $\RR^{m+2}$, $\rho$ is a simply
transitive action. For $\ga=u\cdot t\in \Gamma$, $\ga x=utxt^{-1}$
as above. Then $\rho(\ga x)=utxt^{-1}\cdot 0=utx\cdot 0=\ga\rho(x)$.
So $\rho$ is $\Gamma$-equivariant, $\rho$ induces a diffeomorphism
on the quotients; $\displaystyle U/\Gamma\cong \RR^{m+2}/\Gamma$.

\end{proof}

\begin{pro}\label{keypro}
The fundamental group $\Gamma$ of a
 compact complete Lorentzian similarity manifold $\RR^{m+2}/\Gamma$
admits a nontrivial translation subgroup.
In particular,
the fundamental group of a compact Lorentzian flat space form
admits a nontrivial translation subgroup.
\end{pro}

\begin{proof}Let $\Gamma_0$ be a finite index solvable subgroup of $\Gamma$
and $A(\Gamma_0)=U\cdot T$ the real algebraic hull for $\Gamma_0$ as
above. Let $L:\Gamma_0\ra L(\Gamma_0)$ be the holonomy homomorphism
as in \eqref{discrecase}. As the real algebraic hull for
$L(\Gamma_0)$ can be taken inside ${\rO}(m+1,1)\times\RR^+$, $L$
extends naturally to a homomorphism $L: A(\Gamma_0)\ra
A(L(\Gamma_0))$. We have the following commutative diagram:
\begin{equation}\label{discrecase-1}
\begin{CD}
\RR^{m+2}@>>> {\rSim_L}(\RR^{m+2})@>L>>{\rO}(m+1,1)\times\RR^+\\
@.  @AAA @AAA @. \\
@. A(\Gamma_0)@> L>>A(L(\Gamma_0))\\
@. @AAA @AAA @. \\
@. \Gamma_0@>L>>L(\Gamma_0).\\
\end{CD}
\end{equation}

Suppose that $\RR^{m+2}\cap \Gamma_0=\{1\}$ so that $L:\Gamma_0\ra
L(\Gamma_0)$ is isomorphic. Then $L:A(\Gamma_0)\ra A(L(\Gamma_0))$
is also isomorphic (\cf \cite{ol}). Since $A(\Gamma_0)=U\cdot T$,
this implies $A(L(\Gamma_0))=L(U)\cdot L(T)$. If we note that
$A(L(\Gamma_0))$ is a solvable real linear algebraic group in
${\rO}(m+1,1)\times\RR^+$,
 it follows
\begin{equation}\label{algebraichull}
A(L(\Gamma_0))\leq (\RR^m\rtimes(T^k\times \RR^*))\times \RR^+.
\end{equation} Here $T^k$ is a maximal torus in ${\rm O}(m)$
for which $T^k\times \RR^*$ acts on $\RR^m$ as similarities. As
$L(U)$ is a connected simply connected unipotent Lie group, it
follows $L(U)\leq \RR^m\times \RR^+$. Thus, ${\rm dim}\, L(U)\leq
m+1$. On the other hand, $U/\Gamma_0$ is an $m+2$-dimensional
compact aspherical manifold, we note that ${\rm Rank}\,
\Gamma_0={\rm dim}\, U=m+2$. This contradicts that $L:\Gamma_0\ra
L(\Gamma_0)$ is isomorphic. Therefore $\RR^{m+2}\cap
\Gamma_0\leq\RR^{m+2}\cap \Gamma$ is nontrivial.

\end{proof}

\begin{pro}\label{ls=lf}
Every compact complete Lorentzian similarity manifold
is a Lorentzian flat space form.
\end{pro}

\begin{proof}Consider the exact sequences:
\begin{equation}\label{discrecase-f}
\begin{CD}
1@>>> {\rE}(m+1,1)@>>> {\rSim_L}(\RR^{m+2})@>q>>\RR^+@>>> 1\\
@. @AAA  @AAA @AAA @. \\
1@>>> \Gamma_1@>>> \Gamma@>q>>q(\Gamma)@>>> 1\\
\end{CD}
\end{equation}where $\Gamma_1={\rE}(m+1,1)\cap \Gamma$.
It is enough to show that $q(\Gamma)$ is trivial.
Suppose that there exists an element $\ga=(a,\lam A)\in \Gamma$ such that
$$q(\gamma)=\lam<1.$$
By Proposition \ref{keypro}, let
$\RR^{m+2}\cap \Gamma\cong \ZZ^\ell$ for some $\ell\geq 1$.

Let $\langle\,,\rangle$ be the Lorentz inner product on $\RR^{m+2}$ as before.
\smallskip

\noindent {\bf (1)}\,  Suppose $\ell\geq 1$. Then there exists a vector
$n\in \ZZ^k$ such that $\langle n,n \rangle\neq 0$.
Calculate $$\ga n\ga^{-1}=(a,\lam A)(n,I)(-A^{-1}a,\lam^{-1}A^{-1})
=(\lam An,I)$$ so that $\ga^k n\ga^{-k}=(\lam^k A^kn,I)$.
Take a sequence of orbits at the origin
 $\{\ga^k n\ga^{-k}\cdot 0;\, k=0,1,2,\dots\}$ in $\RR^{m+2}$.
As $\ga^k n\ga^{-k}\cdot 0=\lam^k A^kn$, it follows
$$\langle \lam^k A^kn,\lam^k A^kn\rangle=\lam^{2k}\langle A^kn,A^kn\rangle
=\lam^{2k}\langle n,n\rangle\ra\, 0\ \, (k\ra \infty).$$
Noting $\langle n,n \rangle\neq 0$, this implies that
$\displaystyle \ga^k n\ga^{-k}\cdot 0\ra\, 0$ $(k\ra \infty)$.
As $\Gamma$ acts properly discontinuously,
$\{\ga^k n\ga^{-k}\}$ is a finite set, \ie $\ga^k n\ga^{-k}=1$ for some $k$.
Thus $n=1$ which is a contradiction.\\

\noindent {\bf (2)}\, Suppose $\RR^{m+2}\cap \Gamma\cong \ZZ$ which
is generated by a null vector $n$, \ie $\langle n,n \rangle=0$.
 Since $\Gamma$ leaves $\ZZ$ invariant,
taking a subgroup of index $2$ (if necessary), we may assume
$\displaystyle n=\ga n\ga^{-1}=(\lam An,I)$ for $\ga=(a,\lam A)\in
\Gamma$. This implies $An=\lam^{-1}n$.

Let $\{\ell_1, e_2,\dots,e_{m+1},\ell_{m+2}\}$ be the
basis on $\RR^{m+2}$ such that $$\langle
\ell_1,\ell_1\rangle=\langle\ell_{m+2},\, \ell_{m+2}\rangle=0,\,
\langle e_i,e_j\rangle=\delta_{ij},\,
\langle\ell_1,\ell_{m+2}\rangle= 1.$$
The subgroup ${\Sim}(\RR^m)$ of ${\rm O}(m+1,1)$ has the form
with respect to the above basis:

\begin{equation}\label{similaritygroup}
{\Sim}(\RR^m)=\left\{A=
\left(\begin{array}{ccc}
\lambda^{-1} & x           &-\mbox{\Large $\frac{\lambda|x|^2}{2}$}\\
 \mbox{{\large $0$}}  & B           &-\lambda B\,{}^tx\\
 \mbox{{\large $0$}} &  \mbox{{\large $0$}}                     & \lambda \\
\end{array}\right)\, \mbox{\Huge$|$} \begin{array}{c} \lam\in \RR^+,\\
                                                      B\in\mathrm{O}(m),\\
                                                       x\in\RR^{m}.   \\
                                                      \end{array}
\right\}.
\end{equation}
See \cite{kam6} for details. We may take $n$ for the null vector
$\ell_1$. Since $An=\lam^{-1}n$, $A$ has the form as in
\eqref{similaritygroup}. Then we can write
\begin{equation}\label{matrixform}
\ga=(a,\lam A)= (\left[\begin{array}{c}
a_1 \\
a_2\\
\end{array}\right],
 \left(\begin{array}{ccc}
 1 &\lam x & -\lam^2|x|^2/{2}\\
 0 &\lam B  &-\lam^2 B{}^tx\\
 0 &  0&    \lambda^2 \\
\end{array}\right))
\end{equation}where $a_1\in \RR$, $a_2\in \RR^{m+1}$.
If we put $\displaystyle \rho(\ga)=
(a_2,
 \left(\begin{array}{cc}
\lam B  &-\lam^2 B{}^tx\\
 0&    \lambda^2 \\
\end{array}\right))\in {\rm A}(m+1)$,
then the matrix $\displaystyle  \left(\begin{array}{cc}
\lam B  &-\lam^2 B{}^tx\\
 0&    \lambda^2 \\
\end{array}\right)$ has no eigenvalue $1$ so that
$\rho(\ga)$ has a fixed point $y\in \RR^{m+1}$, \ie
$\rho(\ga)(y)=y$. Conjugate $\Gamma$ by a translation $\displaystyle
t_y= (\left[\begin{array}{c}
  0\\
  -y\\
\end{array}\right], I)$, it follows
\begin{equation}
t_y\ga t_y^{-1}= (\left[\begin{array}{c}
c \\
0\\
\end{array}\right],
 \left(\begin{array}{ccc}
 1 &\lam x & -\lam^2|x|^2/{2}\\
 0 &\lam B  &-\lam^2B{}^tx\\
 0 &  0&    \lambda^2 \\
\end{array}\right))
\end{equation}where $\displaystyle c=a_1+(\lam x,-\lam^2|x|^2/{2})\cdot y\in \RR$.

When we consider the orbits of $\displaystyle\{t_y\ga^k t_y^{-1}\}$ at
the origin $\displaystyle 0=\left(\begin{array}{c}
0 \\
0\\
\end{array}\right)\in \RR^{m+2}$  $(k=1,2,\dots)$, it follows
\begin{equation}\label{matruxc}
t_y\ga^k t_y^{-1}\cdot \left(\begin{array}{c}
0 \\
0\\
\end{array}\right)=
\left(\begin{array}{c}
kc \\
0\\
\end{array}\right).
\end{equation}
On the other hand, noting $t_yn t_y^{-1}=n$, we put
$n=(\left[\begin{array}{c}
t \\
0\\
\end{array}\right],I)$.\\

{\bf Case I.}\, $\displaystyle \frac ct$ is rational, say
$\displaystyle\frac pq$. Take the element $t_y\ga^q t_y^{-1}\cdot
n^{-p}\in t_y\Gamma t_y^{-1}$. Then it follows
\begin{equation}\label{matruxcc}
t_y\ga^q t_y^{-1}\cdot n^{-p}
\left(\begin{array}{c}
0 \\
0\\
\end{array}\right)=
\left(\begin{array}{c}
qc-pt \\
0\\
\end{array}\right)=
\left(\begin{array}{c}
0 \\
0\\
\end{array}\right).
\end{equation}Since $t_y\Gamma t_y^{-1}$ acts freely on $\RR^{m+2}$,
this shows $t_y\ga^q t_y^{-1}\cdot n^{-p}=1$, and thus
$\ga^q=t_y^{-1}n^{p}t_y=n^{p}$. The linear part of $\ga^q$
 is $(\lam A)^q$ for $\ga=(a,\lam A)$, so it follows $(\lam A)^q=I$.
By the formula of \eqref{matrixform}, we obtain $\lam^{2q}=1$.
This is impossible because $\lam<1$ for the element $\ga$.\\

{\bf  Case II.}\, $\displaystyle \frac ct$ is irrational. Let
$\displaystyle\mathop{\lim}_{i\ra \infty}^{}\frac
{\ell_i}{m_i}=\frac ct$, equivalently there exist integers
$m_i,\ell_i$ such that $\displaystyle m_ic-\ell_it\ra\, 0$ $(i\ra
\infty)$. Take a sequence of elements $\{t_y\ga^{m_i} t_y^{-1}\cdot
n^{-\ell_i}\ \, i=1,2,\dots\}$ in $t_y\Gamma t_y^{-1}$ and evaluate
at the origin:
\begin{equation}\label{matruxcc}
t_y\ga^{m_i} t_y^{-1}\cdot n^{-\ell_i}
\left(\begin{array}{c}
0 \\
0\\
\end{array}\right)=
\left(\begin{array}{c}
m_ic-\ell_it \\
0\\
\end{array}\right)\lra \left(\begin{array}{c}
0 \\
0\\
\end{array}\right) \ \ (i\ra \infty).
\end{equation}By properness of $t_y\Gamma t_y^{-1}$,
$\{t_y\ga^{m_i} t_y^{-1}\cdot n^{-\ell_i}\}$ is a finite set,
say $t_y\ga^{m_i} t_y^{-1}\cdot n^{-\ell_i}=
t_y\ga^{m_j} t_y^{-1}\cdot n^{-\ell_j}$  for some $i,j$. As $t_y$ and $n$ commute,
it follows
\begin{equation*}
\ga^m=n^{\ell}\ \ ({}^\exists\, m,\ell\in \ZZ).
\end{equation*}
Again the formula of \eqref{matrixform} implies
$\lam^{2m}=1$ which is impossible for $\ga=(a,\lam A)$.

As a consequence, $q(\Gamma)=\{1\}$ in \eqref{discrecase-f}.

\end{proof}

\section{Lorentzian flat Seifert manifolds}\label{Seifertlo}
Let $M=\RR^{m+2}/\Gamma$ be a compact complete Lorentzian similarity manifold.
It follows from Proposition \ref{keypro} that $\Gamma\cap \RR^{m+2}$ is
nontrivial, say $\ZZ^k$. Then $\Gamma$ normalizes its span $\RR^k$ of $\ZZ^k$
in $\RR^{m+2}$. As $\RR^k$ acts properly on $\RR^{m+2}$ as translations,
we have an equivariant principal bundle:
\begin{equation}\label{Seifertbunlde}
\begin{CD}
(\ZZ^k,\RR^k)@>>> (\Gamma,\RR^{m+2})@>\nu>> (Q,\RR^{\ell})
\end{CD}
\end{equation}where $\ell=m+2-k$ and $Q=\Gamma/\ZZ^k$.
In this case each element $\gamma$ of $\Gamma$ has the form:

\begin{equation}\label{ga-matrix1}
\ga=(\left[\begin{array}{c}
a \\
b\\
\end{array}\right],
\left(\begin{array}{cc}
 A &C\\
 0 &B\\
\end{array}\right))
\end{equation}\mbox{where}
\begin{equation}\label{ga-matrix1}
 \nu(\ga)=
\left(\begin{array}{cc}
 A &C\\
 0 &B\\
\end{array}\right),\  A\in{\rm GL}(k,\ZZ), B\in {\rm GL}(\ell,\RR).
\end{equation}If we put
\begin{equation}\label{affineinduced}
\rho(\nu(\ga))= (b,B) \in {\rA}(\ell),
\end{equation}
 then it is easy to see that $\rho: Q\ra {\rA}(\ell)$ is a
well-defined homomorphism. The quotient group $Q$ acts on $\RR^\ell$
through $\rho$: \[ \al\cdot w=\rho(\nu(\ga))w\, \, \
(\nu(\ga)=\al\in Q, w\in \RR^\ell.)\] Recall the following lemma
(\cf \cite{gk}).

\begin{lemma}\label{properaff} The group $\rho(Q)$ is
a properly discontinuous affine action on $\RR^\ell$
such that
\begin{itemize}
\item ${\rm Ker}\, \rho $ is a finite subgroup.
\item $\RR^{\ell}/\rho(Q)$ is a compact affine orbifold.
\end{itemize}
\end{lemma}

\begin{proof}
 We show that $Q$ acts properly discontinuously.
Consider the pushout:
\begin{equation*}
\begin{CD}
1@>>> \ZZ^k@>>> \Gamma@>\nu>>Q@>>> 1\\
@. @VVV  @VVV @VVV @. \\
1@>>> \RR^k@>>> \RR\cdot\Gamma@>\nu>>Q@>>> 1.\\
\end{CD}
\end{equation*}As both $\RR^k$ and $\Gamma$ act freely and properly
on $\RR^{m+2}$ with $\RR^k/\ZZ^k=T^k$, it follows that
$\RR^k\cdot\Gamma$ acts properly on $\RR^{m+2}$.
 Since
$\displaystyle\RR^k\ra \RR^{m+2}\stackrel{\nu}\lra \RR^\ell$ is a
principal bundle, choose a \emph{continuous} section
$s:\RR^{\ell}\ra \RR^{m+2}$ of $\nu$.
 Let
$\{\al_i\}_{i\in \NN}$ be a sequence of $Q$ such that \[ \al_i\cdot
w_i\ra z,\ w_i\ra w \ \ (i\ra \infty).\] Choose a sequence
$\{\ga_i\}_{i\in \NN}$ from $\Gamma$ such that $\nu(\ga_i)=\al_i$.
As $$\nu(\ga_i s(w_i))=\al_i\cdot w_i=\nu(s(\al_i w_i)),$$ there is
a sequence $\{t_i\}_{i\in \NN}\leq \RR^k$ such that
 \begin{equation*}
t_i\ga_i s(w_i)=s(\al_i w_i), \ \ s(\al_i\cdot w_i)\ra s(z), \ \
s(w_i)\ra s(w).
 \end{equation*} Since
$\RR^k\cdot\Gamma$ acts properly on $\RR^{m+2}$, there is an element
$g\in\RR^k\cdot\Gamma$ such that $t_i\ga_i\ra g$ and so
$\al_i=\nu(t_i\ga_i)\ra \nu(g)\in \Gamma$. Thus $Q$ acts properly
discontinuously on $\RR^\ell$.

 We check that ${\rm Ker}\, \rho $ is finite. Let
$\displaystyle 1\ra \ZZ^k\ra \Gamma_1\ra {\rm Ker}\, \rho \ra 1$ be
the induced extension by the inclusion ${\rm Ker}\,\rho\leq Q$.
 Then $\Gamma_1$ acts invariantly in
the inverse image $\RR^k=\nu^{-1}(pt)$. As $\Gamma$ acts freely and
properly, the quotient $\RR^k/\Gamma_1$ is a closed submanifold in
$M$.  Since $\RR^k/\ZZ^k=T^k$ covers $\RR^k/\Gamma_1$, ${\rm Ker}\,
\rho $ is finite.

\end{proof}

By the definition \cite {lr}, we obtain

\begin{pro}\label{Seifert}
$T^k\ra M\lra \RR^{\ell}/\rho(Q)$ is an injective Seifert fiber
space with typical fiber a torus $T^k$ and
 exceptional fiber a euclidean space form $T^k/F$.
\end{pro}

In \cite{df1} Fried has found all simply transitive Lie group
actions on $4$-dimensional Lorentzian flat space $\RR^4$ which
applied to classify $4$-dimensional compact (complete) Lorentzian
flat manifolds $M$ up to a finite cover. As a consequence, $M$ is
finitely covered by a solvmanifold.

We take a different approach to determine $4$-dimensional compact
complete Lorentzian flat manifolds $M$ from the existence of
\emph{causal actions}.

\begin{definition}
A circle $S^1$ $($respectively $\RR)$ is a causal action on $M$ if
the vector field induced by $S^1$ $($respectively $\RR)$ is causal
$($timelike, spacelike or lightlike$)$ vector field on $M$. Compare
\cite{kam5}.
\end{definition}

We have the following result which occurs particularly in dimension
$4$
 but not in general.
\begin{pro}\label{4-center}
The fundamental group $\Gamma$ of
a compact complete Lorentzian flat manifold $M$
has a finite index subgroup which contains
a central translation subgroup. In particular, some finite cover of $M$ admits a
causal circle action.
\end{pro}

\begin{proof}
Let $\ZZ^k=\Gamma\cap \RR^4$ which is a nontrivial translation
subgroup by Proposition \ref{keypro}. If $k=1$, then $\ZZ$ is
central in a subgroup of finite index in $\Gamma$.

{\bf Case 1.}\, Suppose that $\ZZ^2=\Gamma\cap \RR^4$ (which is maximal).
Let
$$G=\RR^{4}\rtimes (\RR^2\rtimes({\rm SO}(2)\times \RR^+))$$
be the maximal connected solvable Lie subgroup of
${\rm E}(3,1)$. (See the proof of {\bf (2)}
of Proposition \ref{lsim12}.)
 Then
$\Gamma$ lies  in the following exact sequences up to finite index:
{\small
\begin{equation}\label{Ediscrecase}
\begin{CD}
1@>>> \RR^{4}@>>> {\rE}(3,1)@>L>>{\rm O}(3,1)@>>> 1\\
@. @AAA  @AAA @AAA @. \\
1@>>> \ZZ^2@>>> \Gamma@>L>>L(\Gamma)@>>> 1\\
@. @V\mu_PVV  @V\mu_PVV @V\mu_PVV @. \\
1@>>> \RR^{4}@>>> G@>L>>\RR^2\rtimes({\rm SO}(2)\times \RR^+)@>>> 1\\
\end{CD}
\end{equation}
}Here $\mu_P$ is the conjugate homomorphism by some matrix $P\in {\rm GL}(4,\RR)$.
For $\ga\in \Gamma$, we write
\begin{equation}\label{ga-matrix1}
\ga=(\left[\begin{array}{c}
a_1 \\
a_2\\
\end{array}\right],
\left(\begin{array}{cc}
 A &C\\
 0 &B\\
\end{array}\right))
\end{equation}
so that $\displaystyle L(\ga)=\left(\begin{array}{cc}
 A &C\\
 0 &B\\
\end{array}\right)$.
The conjugation homomorphism $\phi:L(\Gamma)\ra {\rm Aut}(\ZZ^2)$ is
given by
\[\phi(L(\ga))=A\in{\rm GL}(2,\ZZ).\]
As $L(\Gamma)$ is a free abelian group of {\rm rank}\, $2$,
$\phi(L(\Gamma))$ belongs to $A$ or $N$ up to conjugacy where ${\rm
SL}(2,\RR)=KAN$. Since ${\rm GL}(2,\ZZ)$ is discrete,
$\phi(L(\Gamma))$ is isomorphic to $\ZZ$, and so ${\rm Ker}\,
\phi=\ZZ$. Choose a generator $\ga_0$ from ${\rm Ker}\, \phi$ and
$\ga\in \Gamma$ for which $\phi(L(\ga))$ generates
$\phi(L(\Gamma))$. Note $\ga_0,\ga$ and $\ZZ^2$ generate $\Gamma$.

Recall the homomorphism $\rho:L(\Gamma)\ra {\rm A}(2)$ defined by
$\displaystyle \rho(L(\ga))=(a_2,B)$ from \eqref{affineinduced}.
Since $\rho(L(\Gamma))$ is a properly discontinuous action of ${\rm
A}(2)$ with compact quotient, the holonomy group of
$\rho(L(\Gamma))$ is a \emph{unipotent subgroup} of ${\rm
GL}(2,\RR)$. In particular, each $B$ has two eigenvalues $1$ and so
$L(\ga)$ has at least two eigenvalues $1$. From \eqref{Ediscrecase},
$\mu_P(L(\Gamma))\leq\RR^2\rtimes({\rm SO}(2)\times \RR^+)$ for
which
\begin{equation}\label{gaform}
\mu_P(L(\ga))=PL(\ga)P^{-1}=
 \left(\begin{array}{ccc}
 \lam^{-1} &x & -\lam|x|^2/{2}\\
 0 &T  &-\lam T{}^tx\\
 0 &  0&    \lambda \\
\end{array}\right)
\end{equation}where $T\in {\rm SO}(2)$.
Since $L(\Gamma)$ is a free abelian group of {\rm rank}\, $2$, it
follows either $\mu_P(L(\Gamma))\leq\RR^2$ or
$\mu_P(L(\Gamma))\leq{\rm SO}(2)\times \RR^+$.

If $\mu_P(L(\Gamma))\leq{\rm SO}(2)\times \RR^+$,
applying $\ga_0\in {\rm Ker}\, \phi$,
\begin{equation}\label{ga0}
PL(\ga_0)P^{-1}=
 \left(\begin{array}{ccc}
 \lam^{-1} &0 & 0\\
 0 &T  &0\\
 0 &  0&    \lambda \\
\end{array}\right).
\end{equation}As $\phi(L(\ga_0))=A=I$ in this case,
$L(\ga_0)$ has all eigenvalues $1$. \eqref{ga0} shows $\lam=1$,
$T=I$. Hence $PL(\ga_0)P^{-1}=I$ or $L(\ga_0)=I$. So $\ga_0\in
\Gamma\cap \RR^4$ which contradicts a maximality of the translation
subgroup $\ZZ^2$. It then follows $\mu_P(L(\Gamma))\leq\RR^2$. In
this case
\begin{equation}\label{ga1}
PL(\ga)P^{-1}=
 \left(\begin{array}{ccc}
 1 &x & -|x|^2/{2}\\
 0 &I  &-{}^tx\\
 0 &  0&  1 \\
\end{array}\right)).
\end{equation}
Then $A$ of \eqref{ga-matrix1}
has two eigenvalues $1$ so $[\ga, \ZZ^2]=(A-I)\ZZ^2$ has rank less
than $2$. Hence there is an element $m\in \ZZ^2$ such that
$[\ga,m]=1$. As $\phi(\ga_0)=1$, $\ga_0m\ga_0^{-1}=m$.
Hence $m$ is a central element of $\Gamma\cap \RR^4$.\\

{\bf Case 2.}\, Suppose that $\ZZ^3=\Gamma\cap \RR^4$.
There is an induced affine action $\rho:L(\Gamma)\ra {\rm A}(1)$ in this case
so that $\rho(L(\Gamma))$ consists of a translation group up to finite index.
As above we obtain
\begin{equation}\label{ga-matrix2}
\ga=(\left[\begin{array}{c}
a \\
b\\
\end{array}\right],
\left(\begin{array}{cc}
 A &C\\
 0 &1\\
\end{array}\right))
\end{equation}where $A\in {\rm GL}(3,\ZZ)$.
Since $L(\ga)$ has the eigenvalue $1$, in view of \eqref{gaform}, it
follows either $T=I$ or $\lam=1$. If $T=I$, $A$ has at least one
eigenvalue $1$. As $\Gamma=\ZZ^3\rtimes \ZZ$, it follows ${\rm
Rank}\,[\gamma,\ZZ^3]<3$. Again there exists an element $n\in\ZZ^3$
such that $\ga n\ga^{-1}=n$. Hence $n$ is a central element in
$\Gamma$.

\end{proof}

Let $\ZZ$ be a central translation subgroup of $\Gamma$.
Put $Q=\Gamma/\ZZ$.
As every element $\ga\in \Gamma$ has the form
\begin{equation}\label{ga-matrix2}
\ga=(\left[\begin{array}{c}
a \\
b\\
\end{array}\right],
\left(\begin{array}{cc}
 1 &C\\
 0 &B\\
\end{array}\right))
\end{equation}where $B\in{\rm GL}(3,\RR)$,
there is an induced action
\[
\varphi: Q\ra {\rm A}(3), \ \varphi(\bar \ga)=(b,B).
\]
Although $\ZZ$ is not necessarily equal to $\Gamma\cap \RR^{4}$, it
can be easily checked that $\varphi: Q\ra {\rm A}(3)$ is a properly
discontinuous action such that $\RR^3/\varphi(Q)$ is compact and
${\rm Ker}\, \varphi$ is finite as in Lemma \ref{properaff}. If
$\RR$ is the span of $\ZZ$ in $\RR^4$, then $\RR$ is causal on
$\RR^4$.

\begin{pro}\label{causalclass}
Every compact complete Lotentzian flat $4$-manifold
admits a causal circle bundle $M$ in its finite cover.

\begin{itemize}
\item[(i)] $S^1$ is a timelike circle. $M=T^4=S^1\times T^3$ where $T^3$ is a
Riemannian flat torus.
\item[(ii)] $S^1$ is a spacelike circle. $(1)$\, $M=S^1\times T^3$,
$(2)$\, $M=S^1\times \mathcal N^3/\Delta$. $(3)$\,
$M=S^1\times \mathcal R/\pi$.  Each $3$-dimensional factor is
a Lorentzian flat manifold.
\item[(iii)] $S^1$ is a lightlike circle. $M=\fS^1\times \mathcal N^3/\Delta$ where
$\displaystyle S^1\ra M\lra \fS^1\times T^2$ is a nontrivial principal
bundle over the affine torus with euler number $k\in \ZZ$. Moreover,
$\fS^1$ is spacelike so $M$ coincides with $(2)$ of case ${\rm (ii)}$.
\end{itemize}

\end{pro}

\begin{proof}
According to whether $\RR$ is timelike or spacelike, we see that the
induced action is Euclidean $\varphi: Q\ra {\rm E}(3)$ or Lorenztian
$\varphi: Q\ra {\rm E}(2,1)$ respectively. Moreover, we have a
decomposition $\RR^4=\RR\times \RR^3$ with respect to the Lorentz
inner product. Then the formula of \eqref{ga-matrix1} becomes:
\begin{equation}\label{ga-matrix2}
\ga=(\left[\begin{array}{c}
a \\
b\\
\end{array}\right],
\left(\begin{array}{cc}
 1 &0\\
 0 &B\\
\end{array}\right)).
\end{equation}For $\varphi(Q)\leq {\rm E}(3)$,
it follows $\displaystyle\varphi(Q)\leq \RR^3$ up to finite index by
the Bieberbach Theorem and hence $\displaystyle
\ga=(\left[\begin{array}{c}
a \\
b\\
\end{array}\right], I)$. As a consequence, $\Gamma\leq \RR^4$.
This shows (i).\\

For $\varphi(Q)\leq {\rm E}(2,1)$, we assume $\varphi(Q)$ is
torsionfree. It is known  that a compact Lorentzian flat
$3$-manifold $\RR^3/\varphi(Q)$ is $T^3$, a Heisenberg nilmanifold
$\mathcal N/\Delta$ or a solvmanifold $\mathcal R/\pi$. (For
example, \cite{fg},\cite{kam6}.) When $\RR^3/\varphi(Q)=\mathcal
N/\Delta$, the center $\RR$ of $\mathcal N$ is the translation
subgroup consisting of $\displaystyle \langle\left[\begin{array}{c}
b_1\\
0\\
0\\
\end{array}\right]\rangle$. The corresponding subgroup $\Delta$ in $\Gamma$
belongs to the translation subgroup $\displaystyle \langle
\Bigl(\left[\begin{array}{c}
a \\
b_1\\
0\\
0\\
\end{array}\right], I\Bigr)\rangle$. It is easy to see that $\Delta$ is a
 central subgroup of {\rm rank}\, $2$.\\
On the other hand, there are two isomorphism classes of
$4$-dimensional (compact) nilmanifolds which are ${\rm
Nil}^4/\Gamma$ or $S^1\times \mathcal N/\Delta$. They are
characterized as  whether the center $C({\rm Nil}^4)=\RR$ or
$C(\RR\times \mathcal N)=\RR^2$. (See \cite{wa} for the
classification of
 $4$-dimensional Riemannian geometric manifolds in the sense of Thurston,
Kulkarni.)
By this classification, $\RR^4/\Gamma=S^1\times \mathcal N/\Delta$.\\

When $\RR^3/\varphi(Q)=\mathcal R/\pi$, it follows that
$[\pi,\pi]=\ZZ^2$. As $\ZZ\leq \Gamma$ is central, it implies
$[\Gamma,\Gamma]=\ZZ^2$. By the classification \cite{wa} of
$4$-dimensional solvmanifolds, the universal covering group $G$ is
either one of solvable Lie groups of Inoue type ${\rm
Sol}_1^4=\mathcal N\rtimes \RR$, ${\rm Sol}_0^4= \RR^3 \rtimes \RR$,
or ${\rm Sol}_{m,n}^4=\RR^3\rtimes \RR$ $(m\neq n)$,
 $\RR\times \mathcal R$ $(m=n)$.
Therefore $[G,G]=\mathcal N$ or $\RR^3$ except for $\RR\times
\mathcal R$. As $[G,G]=[\mathcal R,\mathcal R]=\RR^2$ for $\RR\times
\mathcal R$,
we obtain  $\RR^4/\Gamma=S^1\times \mathcal R/\pi$.\\

We treat the last case that $\RR$ is lightlike. By an ad-hoc
argument or using the result of \cite{df1}, it is shown that
$\Gamma$ is nilpotent with ${\rm Rank}\,C(\Gamma)=2$. So
$\RR^4/\Gamma=\fS^1\times \mathcal N/\Delta$ again. The universal
cover $\fR\times \mathcal N$ is isomorphic to the semidirect product
of the translation subgroup $\RR^3$ with $\RR$; {\small
\begin{equation*}\label{flathomo}
\RR^3= \Bigl(\left[\begin{array}{c}
a\\
b\\
c\\
0\\
\end{array}\right],I\Bigr),\ \,
\RR=\Bigl(\left[\begin{array}{c}
\mbox{\large$\frac{-t^3}6$}\\
\mbox{\large$\frac{-t^2}2$}\\
0\\
t\\
\end{array}\right],
\left(\begin{array}{cccc}
1&t&0&\mbox{\large$\frac{-t^2}2$}\\
0&1&0&-t\\
0&0&1&0\\
0&0&0&1\\
\end{array}\right)\Bigr).
\end{equation*}
} Hence the lightlike action ${\small\displaystyle
\RR=\left[\begin{array}{c}
a\\
0\\
0\\
0\\
\end{array}\right]}$ lies in $\mathcal N$ and
there is another central group ${\small\displaystyle
\fR=\left[\begin{array}{c}
0\\
0\\
c\\
0\\
\end{array}\right]}$ which constitutes a principal bundle and its quotient:
\begin{equation*}\begin{split}
&\RR\ra \fR\times\mathcal N\lra \fR\times \RR^2,\\
&S^1\ra\RR^4/\Gamma\lra \fS^1\times T^2.
\end{split}
\end{equation*}
As $[\Delta,\Delta]=k\ZZ$ $(\forall\, k\in \ZZ)$, $\displaystyle
S^1\ra\mathcal N/\Delta\lra T^2$ is a circle bundle with euler
number $k\in\ZZ$.

\end{proof}

\begin{remark}\label{4nillike}
For the last case, the translation group is the same
$\RR^3=\RR^3\times 0$ but $\RR$ has other possibilities: {\small
\begin{equation*}
\begin{split}
&\RR=\Bigl(\left[\begin{array}{c}
\mbox{\large$\frac{-t^3}6$}\\
0\\
\mbox{\large$\frac{-t^2}2$}\\
t\\
\end{array}\right],
\left(\begin{array}{cccc}
1&0&t&\mbox{\large$\frac{-t^2}2$}\\
0&1&0&0\\
0&0&1&-t\\
0&0&0&1\\
\end{array}\right)\Bigr),\\
&\RR=\Bigl(\left[\begin{array}{c}
\mbox{\large$\frac{-t^3}6$}\\
\mbox{\large$\frac{-t^2}2$}\\
\mbox{\large$\frac{-t^2}2$}\\
t\\
\end{array}\right],
\left(\begin{array}{cccc}
1&t&t&\mbox{\large$\frac{-t^2}2$}\\
0&1&0&-t\\
0&0&1&-t\\
0&0&0&1\\
\end{array}\right)\Bigr).
\end{split}
\end{equation*}
}
\end{remark}

\section{Conformally flat Lorentzian manifold}\label{para}
Recall that the stabilizer of ${\rm PO}(m+2,2)$ at the point
$\hat\infty\in S^{m+1,1}$ is isomorphic to
\[{\rm PO}(m+2,2)_{\hat\infty}=
\RR^{m+2}\rtimes ({\rm O}(m+1,1)\times \RR^+)=
{\rSim_L}(\RR^{m+2}).
\]
Since a maximal amenable subgroup of ${\rm O}(m+1,1)$
 is isomorphic to ${\rm O}(m+1,1)_{\infty}$
 or ${\rm O}(m+1,1)_{0}$,
a maximal amenable Lie subgroup of ${\rm PO}(m+2,2)$ is isomorphic
to either one of the following groups:
\begin{equation}\label{amen}\begin{split}
{\bf (i)}\, \ \ &
\RR^{m+2}\rtimes ({\rm Sim}(\RR^m)\times
\ZZ_2)\times \RR^+.\\
{\bf (ii)}\, \ \ & \RR^{m+2}\rtimes ({\rO}(m+1)\times \ZZ_2)\times
\RR^+.
\end{split}\end{equation}

\begin{definition}\label{FL}
 An $m+2$-manifold
is said to be a Lorentzian parabolic manifold if it admits a
${\Sim}(\RR^m)\times \RR^+$-structure. \end{definition}

As to Case ${\bf (ii)}$, we have

\begin{pro}\label{Lsimilaritymanifold0}
Let $M$ be an $m+2$-dimensional compact conformally flat
Lorentzian manifold whose holonomy group belongs to
$G=\RR^{m+2}\rtimes ({\rO}(m+1)\times \ZZ_2)\times \RR^+$. Then $M$
is finitely covered by the Lorentz model $S^1\times S^{m+1}$, a Hopf
manifold $S^{m+1}\times S^1$, or a torus $T^{m+2}$.
\end{pro}
\begin{proof} There exists a developing pair:
\begin{equation}\label{deve}
\begin{split}
(P\circ \rho,P\circ \dev):(\pi_1(M),\tilde M)&\ra ({\rm
O}(m+2,2)^\sim,\RR\times
S^{m+1})\\
&\ra ({\rm PO}(m+2,2), S^{m+1,1}).
 \end{split}\end{equation}
 By the hypothesis, $\Gamma=P\circ \rho(\pi_1(M))\leq G$. If
$\Gamma$ is a finite subgroup, it follows $\Gamma
\leq{\rO}(m+1)\times \ZZ_2$ so that $P\circ\dev: \tilde M\ra
S^{m+1,1}$ is a covering map. Thus $\displaystyle
(\rho,\dev):(\pi_1(M),\tilde M)\lra (\rho(\pi_1(M)),\RR\times
S^{m+1})$ is an equivariant diffeomorphism. There is a group
extension $\displaystyle 1\ra \ZZ\ra \rho(\pi_1(M))\lra \Gamma\ra 1$
associated to the covering of $\displaystyle \RR\times
S^{m+1}/\ZZ=S^{m+1,1}$. Then $M$ is diffeomorphic to $\RR\times
S^{m+1}/\rho(\pi_1(M))= S^{m+1,1}/\Gamma$.

Suppose that $\Gamma$ is infinite. Recall the equivariant embedding
of $({\rm Sim}_L(\RR^{m+2}),\RR^{m+2})$ into $({\rm
PO}(m+2,2),S^{m+1,1})$ in which $\RR^{m+2}$ is a dense open subset
in $S^{m+1,1}$. The complement $W=S^{m+1,1}-\RR^{m+2}$ consists of
the hypersurface. (See \cite{bcdgm}.) Put ${\rm Dev}=P\circ \dev$
and
\begin{equation}\label{comp}
X=\tilde M - {\rm Dev}^{-1}(W).
\end{equation}Then the developing pair reduces:
\begin{equation}\label{dever}
(\Phi,{\rm Dev}):(\pi,X)\ra (\Gamma,\RR^{m+2})
\end{equation}where $\Gamma\leq G$. Here we put
$\pi=\pi_1(M)$, $\Phi=P\circ \rho$. Since ${\rO}(m+1)\times
\ZZ_2\leq {\rO}(m+2)$, $X/\pi$ is endowed with the usual
similarity structure.\\

\noindent {\bf Case 1.}\ If $X$ is \emph{geodesically complete} with
respect to the pull-back metric of the standard euclidean metric on
$\RR^{m+2}$, then ${\rm Dev}$ is a covering map of $X$ onto
$\RR^{m+2}$ and so ${\rm Dev}$ is a diffeomorphism. Thus $\Gamma$
acts properly discontinuously on $\RR^{m+2}$ so that $\Gamma\leq
\RR^{m+2}\rtimes ({\rO}(m+1)\times \ZZ_2)$, \ie there is no
component in $\RR^+$. $X/\Gamma$ is diffeomorphic to a euclidean
space form $\RR^{m+2}/\Gamma$. Since $\RR^{m+2}$ is dense in
$S^{m+1,1}$, if $\tilde M-X\neq \emptyset$, then ${\rm Dev}:\tilde
M-X\ra {\rm Dev}(\tilde M-X)$ is a homeomorphism. Then $\Gamma$ acts
properly discontinuously on ${\rm Dev}(\tilde M-X)\subset W$. Let
$\Lambda={\rm Dev}(\tilde M-X)$. Since $\Lambda$ is a
$\Gamma$-invariant closed subset (and so compact), every orbit
$\Gamma\cdot x$ for each $x\in \Lambda$ has an accumulation point in
$\Lambda$, so $\Gamma$ cannot act properly on $\Lambda$. Therefore,
$\Lambda={\rm Dev}(\tilde M-X)=\emptyset$ or $\tilde M=X$. Thus $M$
is diffeomorphic to a compact euclidean space form
$\RR^{m+2}/\Gamma$.

Hence $M$ is finitely covered by an $m+2$ -torus
$T^{m+2}=\RR^{m+2}/\ZZ^{m+2}$.\\

\noindent {\bf Case 2.}\ Suppose that a similarity manifold
$X$ is not (geodesically) complete.
It follows from Fried's theorem \cite{df} that there exists a
$\Gamma$-invariant closed (affine) subspace $I$ in $\RR^{m+2}$ which
lies outside the developing image ${\rm Dev}(X)$. (Note that a
similarity manifold $X/\pi$ is not necessarily compact.) In this
case, some element of $\Gamma$ has nontrivial $\RR^+$-summand in
 $G=\RR^{m+2}\rtimes ({\rO}(m+1)\times \ZZ_2\times \RR^+)$. After
conjugation by such element we may assume $0\in I$.

Put the vector subspace $I=\RR^\ell$ in $\RR^{m+2}$ $(\ell< m+2)$.
Since $I$ is closed, the closure $\bar\Gamma\leq G$ leaves the
complement $\RR^{m+2}-\RR^\ell$ invariant. This implies
\begin{equation}\label{group-s}
\begin{CD}
\bar\Gamma\, \leq \, \RR^\ell\rtimes ({\rO}(\ell)\times \RR^+)\times
{\rO}(m-\ell+1)\, \leq\, G\\
 ||@. @.\\
{\rm Sim}(\RR^{\ell}) \times {\rO}(m-\ell+1) @.\\
 \cap@. @.\\
{\rm PO}(\ell+1,1)\times {\rO}(m-\ell+2). @.\\
\end{CD}\end{equation}
Using the real
hyperbolic geometry $({\rm PO}(m+3,1), S^{m+2})$, it can be viewed as

\[\RR^{m+2}-\RR^\ell=S^{m+2}-S^{\ell}=
\HH^{\ell+1}\times S^{m-\ell+1}.\]
The subgroup of ${\rm
PO}(m+3,1)$ preserving this complement
is isomorphic to ${\rm
PO}(\ell+1,1)\times {\rO}(m-\ell+2)$.
Thus $\HH^{\ell+1}\times S^{m-\ell+1}=\RR^{m+2}-\RR^\ell$
admits a complete Riemannian metric
which is invariant under this transitive group of isometries.
In particular any closed subgroup acts properly on
$\RR^{m+2}-\RR^\ell$.

\begin{lemma}[Covering property]\label{complete}
$X$ admits a $\pi$-invariant Riemannian metric such that ${\rm
Dev}:X\ra \RR^{m+2}-\RR^\ell$ is a covering map. \end{lemma}
\begin{proof}
As ${\rm Dev}(X)$ lies outside $I=\RR^\ell$, it restricts the developing
image ${\rm Dev}:X\ra \RR^{m+2}-\RR^\ell$.
Since $\bar\Gamma$ acts properly on $\RR^{m+2}-\RR^\ell$, choose a
$\bar\Gamma$- invariant Riemannian metric on $\RR^{m+2}-\RR^\ell$
such that ${\rm Dev}:X\ra \RR^{m+2}-\RR^\ell$ is a local isometry with
respect to the pullback metric of $\RR^{m+2}-\RR^\ell$.
Let $P:\tilde M\ra M$ be the covering projection.
 As the pullback metric on $X$ is $\pi$-invariant, the (restricted)
 projection $P:X\ra X/\pi$ induces a Riemannian metric on $X/\pi$.

Let $\{x_j\}$ be a Cauchy sequence in $X/\pi$. Since $X/\pi\subset
M$ which is compact, $\displaystyle\mathop{\lim}_{j\ra \infty}
{x_j}=w\in M$. Choose a point $\tilde w\in \tilde M$ and
neighborhoods $U(\tilde w)\subset \tilde M$, $U(w)\subset M$ such
that $P :U(\tilde w)\ra U(w)$ is a homeomorphism with $P(\tilde
w)=w$. Let $\{\tilde x_j\}\subset U(\tilde w)$ be a sequence such
that $P(\tilde x_j)=x_j$ and $\displaystyle\mathop{\lim}_{j\ra
\infty}{\tilde x_j}=\tilde w$. As $P\colon U(\tilde w)\cap X\ra
U(w)\cap X/\pi$ is an isometry, $\{\tilde x_j\}$ is also Cauchy.
 Since the sequence $\{{\rm Dev}(\tilde x_j)\}$ is Cauchy
in $\RR^{m+2}-\RR^\ell$ where $\RR^{m+2}-\RR^\ell$ is complete,
$\displaystyle\mathop{\lim}_{j\ra \infty} {\rm Dev}(\tilde x_j)=a\in
\RR^{m+2}-\RR^\ell$. As $\displaystyle\mathop{\lim}_{j\ra
\infty}{\tilde x_j}=\tilde w$, ${\rm Dev}(\tilde w)=a$. So $\tilde
w\in X$ (because $\tilde M-X={\rm Dev}^{-1}(W)$ and $a\notin
W=S^{m+1,1}-\RR^{m+2}$) and hence $P(\tilde w)=w\in X/\pi$, $X/\pi$
is complete. So $X$ is complete, ${\rm Dev}:X\ra \RR^{m+2}-\RR^\ell$
is a covering map.

\end{proof}

The proof of Lemma \ref{complete}
works when $\RR^\ell$ is replaced by the following space $Y$.

\begin{pro}\label{comleteremark}
Let $Y$ be a $\Gamma$-invariant closed subset such that
the complement $\RR^{m+2}-Y$ admits a $\Gamma$-invariant
complete Riemannian metric.
If $(\Psi, {\rm Dev}) : (\pi,X)\ra (\Gamma,\RR^{m+2}-Y)$
 is a developing pair, then
${\rm Dev} : X\ra \RR^{m+2}-Y$ is a covering map.
\end{pro}

From Lemma \ref{complete}, if $\ell\neq m$, ${\rm Dev}:X\ra
\RR^{m+2}-\RR^\ell$ is a homeomorphism so $\Gamma$ is discrete. If
we recall that $\Gamma$ has a nontrivial summand in $\RR^+$ ({\bf
Case (2)}), \eqref {group-s} implies
\begin{equation}\label{lnotm}
\Gamma \leq {\rO}(\ell)\times
\RR^+ \times {\rO}(m-\ell+1)\leq {\rO}(m+2)\times\RR^+.
\end{equation}\\

If $\ell=m$, then ${\rm Dev}:X\ra \RR^{m+2}-\RR^m=
\HH^{m+1}_\RR\times S^{1}$ is a covering map such that
$\displaystyle \Gamma\leq {\rm Sim}(\RR^{m}) \times{\rO}(1)$ by
Lemma \ref{complete}. Let $p:\HH^{m+1}_\RR\times\RR^{1}\ra
\HH^{m+1}_\RR\times S^1$ be the projection. If $\widetilde {\rm
Dev}:X\ra \HH^{m+1}_\RR\times \RR^{1}$ is a lift of ${\rm Dev}$,
then it is a diffeomorphism so that the conjugate group $\tilde
\Gamma=\widetilde {\rm Dev}\circ \pi\circ\widetilde{\rm Dev}^{-1}$
acts properly discontinuously on $\HH^{m+1}_\RR\times \RR^{1}$.
Moreover, associated with the infinite covering of
$\HH^{m+1}_\RR\times S^{1}$, there is the commutative diagram:
{\scriptsize
\begin{equation*}\label{group-dis}
\begin{CD}
1@>>> \ZZ@>>>{\rm Sim}(\RR^{m})\times (\ZZ\rtimes {\rO}(1))
@>p>>{\rm Sim}(\RR^{m})\times{\rO}(1)@>>> 1\\
@.@. @AAA @AAA \\
@.@. \tilde \Gamma @>p>> \Gamma@>>> 1\\
\end{CD}\end{equation*}
}

Since $\tilde \Gamma$ is discrete and has a nontrivial summand in $\RR^+$
(because so is $\Gamma$), it follows
$\tilde\Gamma\leq {\rO}(m)\times \RR^+\times (\ZZ\rtimes {\rO}(1))$
which shows
\begin{equation}\label{lnotm1}
\Gamma\leq {\rO}(m)\times \RR^+\times {\rO}(1)\leq {\rO}(m+2)\times\RR^+.
\end{equation}
For both of \eqref{lnotm}, \eqref{lnotm1}, $\Gamma$ fixes $0$ such
that the complement $\RR^{m+2}-\{0\}= S^{m+1}\times\RR^+$ admits a
complete Riemannian metric invariant under ${\rO}(m+2)\times\RR^+$.
Applying Proposition \ref{comleteremark},
 $(\Phi, {\rm Dev}) : (\pi,X)\ra\, (\Gamma,\RR^{m+2}-\{0\})$
  is an equivariant covering map. Hence
${\rm Dev}: X\ra\, \RR^{m+2}-\{0\}$ is a diffeomorphism. On the
other hand, we can show that $\Lambda={\rm Dev}(\tilde
M-X)=\emptyset$ as in the argument of {\bf Case 1}, ${\rm
Dev}:\tilde M\ra \RR^{m+2}-\{0\}$ is a diffeomorphism. Hence $M$ is
finitely covered by a Hopf manifold $S^{m+1}\times S^{1}$. In fact,
$\displaystyle
M\cong\RR^{m+2}-\{0\}/\Gamma=S^{m+1}\times\RR^+/\Gamma=
S^{m+1}\mathop{\times} S^{1}/F$ where $F$ is a finite group of
$({\rO}(m)\times \ZZ_2)\times S^1$ acting freely.

\end{proof}

\begin{theorem}\label{Lsimilaritymanifold1}
Let $M$ be an $m+2$-dimensional compact conformally flat Lorentzian
manifold whose holonomy group is a virtually solvable subgroup lying
in ${\rSim_L}(\RR^{m+2})$.
 Then $M$ is either a conformally flat Lorentzian
parabolic manifold or finitely covered by the Lorentz model
 $S^1\times S^{m+1}$, a Hopf manifold $S^{m+1}\times S^1$,
or a torus $T^{m+2}$.
\end{theorem}

\begin{proof}
Given a compact conformally flat Lorentzian $(m+2)$-manifold $M$,
there exists a developing pair
\begin{equation}\label{deve}
\begin{split}
(P\circ \rho,P\circ \dev):(\pi_1(M),\tilde M)&\ra \,({\rm
O}(m+2,2)^\sim,\RR\times
S^{m+1})\\
&\ra\, ({\rm PO}(m+2,2), S^{m+1,1}). \end{split}\end{equation}
 Denote ${\Aut}(T_{\hat\infty}S^{m+1,1})$
the automorphism group of $T_{\hat\infty}S^{m+1,1}$ where
$T_{\hat\infty}S^{m+1,1}$ is the tangent space of $S^{m+1,1}$ at
${\hat\infty}$. Let $L:{\rSim_L}(\RR^{m+2})= {\rm
PO}(m+2,2)_{\hat\infty}\ra\, {\rm O}(m+1,1)\times \RR^+$ be the
projection as before such that $\displaystyle {\rm O}(m+1,1)\times
\RR^+\leq \mathop{\Aut}(T_{\hat\infty}S^{m+1,1})$. As
$\Gamma=P\circ\rho(\pi_1(M))$ is virtually solvable in
${\rSim_L}(\RR^{m+2})$, there are two possibilities {\bf (i)}, {\bf
(ii)} as in \eqref{amen}, \ie the structure group $L(\Gamma)$
belongs to either ${\rm Sim}(\RR^m)\times \ZZ_2\times \RR^+$ or
${\rO}(m+1)\times \ZZ_2\times \RR^+$. By Definition \ref{FL} (\cf
\cite{kam6}), the case {\bf (i)} implies that $M$ is a conformally
flat Lorentzian parabolic manifold. For the case {\bf (ii)}, it
follows $\Gamma\leq \RR^{m+2}\rtimes ({\rO}(m+1)\times \ZZ_2)\times
\RR^+$. Hence the assertion follows from Proposition
\ref{Lsimilaritymanifold0}.

\end{proof}

\begin{remark}\label{simcomplete} We collect several remarks and problems.
\begin{itemize}
\item[(i)] If $M$ is a compact Lorentzian similarity
manifold with virtually solvable holonomy group, then it is easy to
see that $M$ is either a Lorentzian parabolic similarity manifold, a
euclidean space form or a Hopf manifold.
\item[(ii)] As a compact Lorentzian flat manifold is complete by Carriere's
celebrated theorem \cite{ca}, it is a Lorentzian parabolic
similarity manifold by the definition.
\item[(iii)] There is a compact
incomplete Lorentzian similarity $m+2$-manifold whose fundamental
group is isomorphic to $\Gamma\times\ZZ$ where $\Gamma$ is a
torsionfree discrete cocompact isometry subgroup of the hyperboloid
$\HH^{m+1}_\RR$. In particular, the virtual solvability of
$\pi_1(M)$ does not follow from compactness for a Lorentzian
similarity manifold $M$.
\item[(iv)]
 Let $M$ be a compact Lorentzian parabolic similarity
manifold with virtually solvable holonomy group. \emph{Is $M$
complete?} We don't know whether there exists a compact Lorentzian
parabolic similarity manifold other than compact Lorentzian flat
manifolds. See Corollary $\ref{Fpara}$ for compact Fefferman-Lorentz
parabolic similarity manifold.

\end{itemize}
For ${\rm (iii)}$, this is easily obtained by taking the interior of
the cone in $\RR^{m+2}$ which is identified with the product
$\HH^{m+1}_\RR \times \RR^+$ on which the holonomy group
${\rO}(m+1,1)\times \RR^+$ acts transitively.
\end{remark}

\section{Fefferman-Lorentz parabolic structure}\label{parabolic}

Let $\ZZ_2$ be the subgroup of the center $S^1$ in ${\rm U}(n+1,1)$.
Put $\hat{\rm U}(n+1,1)={\rm U}(n+1,1)/\ZZ_2$. The inclusion
$\displaystyle {\rm U}(n+1,1)\ra {\rm O}(2n+2,2)$ defines a natural
embedding $\displaystyle\hat{\rm U}(n+1,1)\ra {\rm PO}(2n+2,2)$.
Then $\hat{\rm U}(n+1,1)$ acts transitively on $S^{2n+1,1}$ so that
$(\hat{\rm U}(n+1,1), S^{2n+1,1})$ is a subgeometry of $({\rm
PO}(2n+2,2), S^{2n+1,1})$.

As in Introduction, a conformally flat \emph{Fefferman-Lorentz
parabolic} manifold $M$ is a $2n+2$ -dimensional smooth manifold
locally modelled on the geometry $({\rm U}(n+1,1),S^1\times
S^{2n+1})$. See \cite{kam6} for details. We observe which subgroup
in ${\Sim}_L(\RR^{2n+2})$ corresponds to \emph{conformally flat
Fefferman-Lorentz parabolic structure}. Let $q: S^{2n+1,1}\ra
S^{2n+1}$ be the projection and $\{\hat\infty\}$ the infinity point
of $S^{2n+1,1}$ which maps to $\{\infty\}$ of $S^{2n+1}$. As a
spherical $CR$-manifold, $S^{2n+1}-\{\infty\}$ is identified with
the Heisenberg Lie group $\mathcal N$. Since the stabilizer is
\[{\rm PO}(2n+2,2)_{\hat\infty}=\RR^{2n+2}\rtimes
 ({\rm O}(2n+1,1)\times \RR^+)
={\Sim}_L(\RR^{2n+2}),\]
the intersection $\hat{\rm U}(n+1,1)\cap{\rm
PO}(2n+2,2)_{\hat\infty}$ becomes $$\hat{\rm
U}(n+1,1)_{\hat\infty}=\mathcal N\rtimes ({\rm U}(n)\times \RR^+).$$

Noting $\displaystyle{\Sim}^*(\RR^{2n})=\RR^{2n}\rtimes ({\rm
O}(2n)\times\RR^*)\leq {\rm O}(2n+1,1)$, it follows
\begin{equation}
\begin{split}
\mathcal N\rtimes ({\rm U}(n)\times \RR^+)&\leq
\RR^{2n+2}\rtimes({\Sim}^*(\RR^{2n})\times \RR^+)\\
&=(\RR^{2n+2}\rtimes \RR^{2n})\rtimes ({\rm O}(2n)\times
\RR^*)\times \RR^+
\end{split}
\end{equation}where $\RR^{2n+2}\rtimes \RR^{2n}$ is a
nilpotent Lie group such that $\mathcal N\leq\RR^{2n+2}\rtimes
\RR^{2n}$. We have shown in \cite{kam6} (Compare \cite{fe}.)
\begin{theorem}\label{CR=Weyl}
A Fefferman-Lorentz manifold $S^1\times N$ is conformally flat
 if and only if $N$ is a spherical $CR$-manifold.
\end{theorem}
Note that $S^1$ acts as lightlike isometries on Fefferman-Lorentz
manifolds $S^1\times N$ so does its lift $\RR$ on $\RR\times N$. If
$({\rm U}(n+1,1)^{\sim},\RR\times S^{2n+1})$ is an infinite covering
of $(\hat{\rm U}(n+1,1), S^{2n+1,1})$, then the subgroup $\RR\times
(\mathcal N\rtimes {\rm  U}(n))$ of ${\rm U}(n+1,1)^{\sim}$ acts
transitively on the complement $\RR\times
S^{2n+1}-\RR\cdot{\infty}=\RR\times \mathcal N$. If $\ZZ\times
\Delta$ is a discrete cocompact subgroup of $\RR\times (\mathcal
N\rtimes {\rm  U}(n))$, then we obtain (\cf \cite{kam6})
\begin{pro}\label{prabolicnil}
$S^1\times \cN/\Delta$ is a conformally flat Lorentzian parabolic
manifold on which $S^1$ acts as lightlike isometries.
\end{pro}

\begin{remark}\label{4dimclass}
In ${\rm (iii)}$ of Proposition $\ref{causalclass}$, we saw that a
finite cover of a compact (complete) Lotentzian flat $4$-manifold
admitting a lightlike circle $S^1$ is the nilmanifold $\fS^1\times
\mathcal N^3/\Delta$ with nontrivial circle bundle $S^1\ra
\fS^1\times \mathcal N^3/\Delta\ra \fS^1\times T^2$. The circle
$\fS^1$ acts as spacelike isometries. Therefore, the $4$-nilmanifold
$S^1\times \mathcal N^3/\Delta$ of Proposition $\ref{prabolicnil}$
is not conformal to a Lorentzian flat manifold. In fact, if it
admits a Lorentzian flat structure within the conformal class, $S^1$
would be spacelike as above. But $S^1$ is still lightlike under the
conformal change of the Lorentzian metric, being contradiction.

\end{remark}

\section{Developing maps}\label{holonomydiscrete}

 Suppose that $M$ is a
$2n+2$-dimensional conformally flat Fefferman-Lorentz parabolic
manifold. There is a developing pair:
\begin{equation}\label{devpairS}
(\rho,\dev):(\pi,\tilde M)\ra ({\rm
U}(n+1,1)^{\sim},\tilde S^{2n+1,1}).
\end{equation}
Let
\begin{equation}\label{twomaps}
\begin{split}
q:&({\rm U}(n+1,1)^{\sim},\tilde S^{2n+1,1})\ra (\hat{\rm
U}(n+1,1), S^{2n+1,1}),\\
p:&(\hat{\rm U}(n+1,1), S^{2n+1,1})\ra ({\rm
PU}(n+1,1),S^{2n+1})\\
\end{split}
\end{equation}
be the equivariant projections. Let $\Gamma=\rho(\pi)$ be the
holonomy group of $M$ as before. There is a central group extension:
\begin{equation}\label{exatgamma}
1\ra \RR\ra {\rm U}(n+1,1)^{\sim}\stackrel{p\circ q}\lra {\rm
PU}(n+1,1)\ra 1.\end{equation}

\begin{theorem}\label{th:discrete}
Let $M$ be a compact conformally flat Fefferman-Lorentz parabolic
manifold in dimension $2n+2$. Suppose that the holonomy group
$\Gamma$ is discrete. If the developing map ${\dev}:\tilde
M\ra\tilde S^{2n+1,1}=\RR\times S^{2n+1}$ misses a closed subset
which is invariant under $\RR$ and $\Gamma$,
 then ${\dev}$ is a covering map onto the image.
\end{theorem}

\begin{proof}
Let $\Lambda$ be both $\RR$ and $\Gamma$-invariant closed subset
such that ${\dev}(\tilde M)\subset \tilde S^{2n+1,1}-\Lambda$.

\vskip0.1cm {\bf I}. Suppose that $p\circ q(\Lambda)$ contains more
than one point in $S^{2n+1}$.  Let $\fL(G)$ be the \emph{limit set}
for a hyperbolic group $G$ (\cf \cite{cg}). As $p\circ q(\Lambda)$
is invariant under $p\circ q(\Gamma)$, \emph{Minimality} of
\emph{limit set} implies that $\fL(p\circ q(\Gamma))\subset p\circ
q(\Lambda)$. In particular, $(p\circ q)^{-1}(\fL(p\circ
q(\Gamma)))\subset \RR\cdot \Lambda=\Lambda$. It follows
\begin{equation}\label{devmis}
{\dev}:\tilde M\ra \tilde S^{2n+1,1}-(p\circ q)^{-1}(\fL(p\circ
q(\Gamma))).
\end{equation}

{\bf (i)}\, If $p\circ q(\Gamma)$ is discrete, then $p\circ
q(\Gamma)$ acts properly discontinuously on \emph{the domain of
discontinuity} $S^{2n+1}- p\circ q(\Lambda)$. It is easy to see that
the closure $\bar\Gamma\leq {\rm U}(n+1,1)^{\sim}$ acts properly on
$\tilde S^{2n+1,1}- \Lambda$. Since $\Gamma$ is discrete by the
hypothesis, $\Gamma$ acts properly discontinuously on $\tilde
S^{2n+1,1}- \Lambda$ so there exists a $\Gamma$-invariant Riemannian
metric. (Compare \cite{ko} for instance.) As ${\dev}:\tilde M\ra
\tilde S^{2n+1,1}-\Lambda$ is an immersion, the pullback metric by
${\dev}$ is a $\pi$-invariant Riemannian metric on $\tilde M$. Thus
${\dev}:\tilde M\ra  \tilde S^{2n+1,1}-\Lambda$ is a covering map.\\

We have a commutative diagram of group extensions from
\eqref{exatgamma}:
\begin{equation}\label{exatgamma1}
\begin{CD}
1@>>> \RR@>>>{\rm U}(n+1,1)^{\sim}@>{p\circ q}>> {\rm PU}(n+1,1)@>>> 1\\
@. ||@. @AAA @AAA @.\\
1@>>> \RR@>>>\RR\cdot \Gamma@>{p\circ q}>> p\circ q(\Gamma)@>>>1
\end{CD}
\end{equation}Here $\RR\cdot \Gamma$ is the pushout.\\

{\bf (ii)}\, Suppose that $p\circ q(\Gamma)$ is not discrete, then
the identity component of the closure $\overline{p\circ
q(\Gamma)}^0$ is solvable by Bieberbach-Auslander's theorem
\cite[8.24 Theorem]{ra}. We may assume that $\overline{p\circ
q(\Gamma)}^0$ is noncompact, so it follows up to conjugacy
$$\overline{p\circ
q(\Gamma)}^0\leq {\rm PU}(n+1,1)_\infty= \cN\rtimes ({\rm
U}(n)\times \RR^+).$$ As the normalizer of $\overline{p\circ
q(\Gamma)}^0$ is also contained in $\cN\rtimes ({\rm U}(n)\times
\RR^+)$, we have $p\circ q(\Gamma)\leq \cN\rtimes ({\rm U}(n)\times
\RR^+)$. Hence \eqref{exatgamma1} shows that $\Gamma\leq \RR\cdot
\cN\rtimes ({\rm U}(n)\times \RR^+)$. If we note that $\RR^+$ acts
as the multiplication $$\lam(a,z)=(\lam^2\cdot a,\lam\cdot z)$$ for
$\lam\in \RR^+, (a,z)\in \cN$ (\cf \cite{kam1}). Since $\Gamma$ is
discrete, it is easy to check
\begin{equation}\label{twocases}
\begin{split}
&\Gamma\leq \RR\times ({\rm U}(n)\times \RR^+)\, \, \mbox{when}\,
\Gamma\, \mbox{is nontrivial in}\,
\RR^+,\, \mbox{otherwise}\\
& \Gamma\leq \RR\cdot \cN\rtimes {\rm U}(n).
\end{split}\end{equation}
Then it follows that $\fL(p\circ q(\Gamma))\subset \fL({\rm
U}(n)\times \RR^+)=\{0,\infty\}$, $\fL(p\circ q(\Gamma))\subset
\fL(\cN\rtimes {\rm U}(n))=\{\infty\}$ respectively. Thus $(p\circ
q)^{-1}(\fL(p\circ q(\Gamma)))=\RR\cdot \{0,\infty\}$, $(p\circ
q)^{-1}(\fL(p\circ q(\Gamma)))=\RR\cdot \{\infty\}$ respectively. We
obtain

\begin{itemize}
\item
${\dev}:\tilde M\ra \tilde S^{2n+1,1}-\RR\cdot
\{0,\infty\}=\RR\times (S^{2n}\times\RR^+)$ which is a
diffeomorphism. $M$ is diffeomorphic to $\RR\times
(S^{2n}\times\RR^+)/\Gamma$ and so $M$ is finitely covered by
$S^1\times S^{2n}\times S^1$.
\item
${\dev}:\tilde M\ra \tilde S^{2n+1,1}-\RR\cdot \{\infty\}=\RR\times
(S^{2n+1}-\{\infty\})=\RR^+\times \cN$ which is a diffeomorphism.
$M$ is diffeomorphic to $\RR\times \cN/\Gamma$ so that
 $M$ is finitely covered by $S^1\times \cN/\Delta$.
\end{itemize}In the first case, it follows $p\circ
q(\Lambda)=\{0,\infty\}$. For the second case, $p\circ
q(\Lambda)=\{\infty\}$ which is excluded by the assumption of Case
${\bf I}$.

\vskip0.2cm {\bf II}. Suppose that $p\circ q(\Lambda)$ consists of a
single point, say $\{\infty\}\in S^{2n+1}$. As
$\Lambda=\RR\cdot\infty$, it follows ${\dev}:\tilde M\ra \tilde
S^{2n+1,1}-\RR\cdot\{\infty\}=\RR\times\cN$. Since $p\circ
q(\Gamma)$ fixes $\{\infty\}$, $p\circ q(\Gamma)\leq {\rm
PU}(n+1,1)_\infty= \cN\rtimes ({\rm U}(n)\times\RR^+)$. As in the
argument of ${\bf (ii)}$, it follows either $(1)$ $\Gamma\leq
\RR\cdot \cN\rtimes {\rm U}(n)$ or $(2)$ $\Gamma\leq \RR\times({\rm
U}(n)\times \RR^+)$ (\cf \eqref{twocases}).\\

For $(1)$, $\RR\times\cN$ admits an $\RR\cdot \cN\rtimes {\rm
U}(n)$-invariant Riemannian metric so ${\dev}:\tilde M\ra
\RR\times\cN$ is a diffeomorphism. Note that $M$ is diffeomorphic to
$\RR\times \cN/\Gamma$ whose finite cover $S^1\times \cN/\Delta$ is
 a conformally
flat Lorentzian parabolic manifold with virtually nilpotent
fundamental group.
\\

Suppose $(2)$ where $\Gamma\leq \RR\times({\rm U}(n)\times \RR^+)$.
As $\RR\times({\rm U}(n)\times \RR^+)$ leaves $\RR\times\{0\}$
invariant, put $X=\tilde M-{\dev}^{-1}(\RR\times\{0\})$ which is
invariant under $\RR\times({\rm U}(n)\times \RR^+)$. This induces a
developing map ${\dev}:X\ra \RR\times(\cN-\{0\})=\RR\times
(S^{2n}\times \RR^+)$. Since $\RR\times (S^{2n}\times \RR^+)$ admits
a complete Riemannian metric invariant under $\RR\times({\rm
U}(n)\times \RR^+)$, the same proof of Proposition
\ref{comleteremark} implies that
 ${\dev}:X\ra \RR\times(\cN-\{0\})$
is a (covering) diffeomorphism. If ${\dev}^{-1}(\RR\times\{0\})\neq
\emptyset$, then ${\dev}:\tilde M\ra {\dev}(\tilde
M)\subset\RR\times\cN$ is also a diffeomorphism. As $\Gamma$ acts
properly on $\RR\times\cN$, it follows ${\dev}(\tilde
M)=\RR\times\cN$. But $\Gamma$ has cohomological dimension at most
$2$, this cannot occur. Then ${\dev}^{-1}(\RR\times\{0\})=\emptyset$
which concludes that ${\dev}:\tilde M\ra \RR\times(\cN-\{0\})$ is a
diffeomorphism. In this case  $p\circ
q(\Lambda)=\{\infty\}\subset\{0,\infty\}$. This finishes the proof
of the theorem.

\end{proof}

\begin{remark}\label{devcase}
According to the cases ${\bf I}$-${\bf {\rm (i)}}$, ${\bf I}$-${\bf
{\rm (ii)}}$, ${\bf II}$-${\rm (1)}$ and ${\bf II}$-${\rm (2)}$, the
following occurs:
\begin{itemize}
\item[(a)] ${\dev}:\tilde M\ra \tilde S^{2n+1,1}-\Lambda$
is a covering map in which $\#p\circ q(\Lambda)\geq 2$.
\item[(b)] ${\dev}:\tilde M\ra \tilde S^{2n+1,1}-\RR\cdot \{0,\infty\}$
is a diffeomorphism in which $p\circ q(\Lambda)=\{0,\infty\}$.
\item[(c)] ${\dev}:\tilde M\ra \tilde S^{2n+1,1}-\RR\cdot \{\infty\}$ is a
diffeomorphism in which $p\circ q(\Lambda)=\{\infty\}$
\item[(d)] ${\dev}:\tilde M\ra \tilde S^{2n+1,1}-\RR\cdot \{0,\infty\}
=\RR\times (\mathcal N-\{0\})$ is a diffeomorphism in which $p\circ
q(\Lambda)=\{\infty\}$.
\end{itemize}
\end{remark}

\begin{corollary}\label{Fpara}
There exists no $2n+2$ -dimensional compact Fefferman-Lorentz
parabolic similarity manifold with discrete holonomy group.
\end{corollary}

\begin{proof}
Recall that there is an equivariant embedding of $\RR^{m+2}$ into
$S^{m+1,1}$ with respect to ${\rm Sim}_L(\RR^{m+2})=\RR^{m+2}\rtimes
({\rm O}(m+1,1)\times \RR^+)={\rm PO}(m+2,2)_{\hat\infty}$:
\begin{equation}\label{Lorentembedd}
\iota: (x,y)\ra \left[\frac{|x|^2-y^2}2-1,\sqrt 2x,\sqrt 2
y,\frac{|x|^2-y^2}2+1\right]
\end{equation}for
$x=(x_1,\cdots,x_{m+1})$ and $|x|=\sqrt {x_1^2+\cdots+x_{m+1}^2}$.
For $m=2n$, let $\displaystyle \hat\infty=[1,0,\dots,0,1]\in
S^{2n+1,1}$. (See \cite{kam1}.) Then $\RR^{2n+2}$ misses
$\hat\infty$ in $S^{2n+1,1}$. Moreover, the orbit of $S^1\,(={\rm
SO}(2))$-action at $\hat\infty\in S^{2n+1,1}$ becomes \[
S^1\cdot\hat\infty=\{
[\cos\theta,\sin\theta,0,\dots,0,-\sin\theta,\cos\theta],\,\
\theta\in \RR\}.\] In view of the formula \eqref{Lorentembedd}, it
follows
\begin{equation}\label{equiembeddingR}\begin{split}
\RR^{2n+2}&\subset S^{2n+1,1}-S^1\cdot\hat\infty\\
          &=S^1/\ZZ_2\times (S^{2n+1}-\{\infty\})=S^1\times \cN.
\end{split}
\end{equation}If we put $\mathcal I=S^1-\{\infty\}$, then note
\begin{equation}\label{conf}
\RR^{2n+2}=\cI\times \cN.
\end{equation}
Putting $\Gamma=\rho(\pi)$, the developing pair reduces:
\begin{equation}\label{first-dev}
(\rho,{\dev}):(\pi,\tilde M)\ra (\Gamma,\RR^{2n+2})\subset ({\rm
U}(n+1,1)^{\sim},\RR\times \cN).
\end{equation}

Let $q\circ\dev: \tilde M\ra \RR^{2n+2}$ be the developing map for
which
 $q(\Gamma)\leq \hat{\rm U}(n+1,1)$.
Then  $\dev$ misses $\Lambda=q^{-1}(S^1\cdot\hat\infty)$ which
 is invariant under both $\Gamma$ and $\RR$.
 In particular, $p\circ q(\Lambda)=\{\infty\}$.
As $\Gamma$ is discrete in ${\rm U}(n+1,1)^{\sim}$ by the
hypothesis, we can apply Theorem \ref{th:discrete} to show that
either $(c)$ or $(d)$ of Remark \ref{devcase} occurs.

According to $(c)$ or $(d)$, it follows either $\Gamma\leq \RR\times
(\mathcal N\rtimes {\rm U}(n))$ or $\Gamma\leq \RR\times( {\rm
U}(n)\times \RR^+)$. However, $\Gamma$ leaves $\RR^{2n+2}$
invariant. As the developing image is connected, we note by
\eqref{conf} that $\dev(\tilde M)\subset \cI\times \cN\subset
\RR\times \cN$. Here $\cI$ is one of the components $\ZZ\cI\subset
\RR$. This implies $\Gamma\leq\mathcal N\rtimes {\rm U}(n)$ or
$\Gamma\leq {\rm U}(n)\times \RR^+$ respectively. Then \eqref
{first-dev} becomes:
\begin{equation*}\label{secdev}
\begin{CD}
(\pi,\tilde M)@>(\rho,{\dev})>> (\Gamma,\cI\times \cN)\subset (
\mathcal N\rtimes {\rm U}(n),\RR\times \cN),\ \ \ \ \ \\
(\pi,\tilde M)@>(\rho,{\dev})>> (\Gamma,\cI\times (\cN-
\{0\}))\subset({\rm U}(n)\times \RR^+,{\cI}\times S^{2n}\times
\RR^+).
\end{CD}
\end{equation*}
 It follows that $M\cong \cI\times \cN/\Gamma$,
or $M\cong {\cI}\times (S^{2n}\times S^1/F)$ respectively. In each
case, $M$ cannot be compact.

\end{proof}

\begin{remark}
 The hypothesis that $\Gamma$ is discrete is
used to eliminate Case ${\bf II}$ that the limit set consists of a
single point. Concerned with the hypothesis on Theorem
$\ref{th:discrete}$, \emph{discreteness} of the holonomy group and
that $\Lambda$ is $\RR$-invariant may be dropped. More generally we
pose

\begin{conjecture}
Given a compact conformally flat Lorentzian manifold,
if a developing map is not surjective, then it is a covering map onto the image.
\end{conjecture}
\end{remark}


\end{document}